\documentclass{gtpart}   
\usepackage{pinlabel,amscd}
\usepackage[all]{xy}

\title[On uniqueness of differential structures on orbifolds]{On uniqueness of differential 
structures on orbifolds}

\author{Marja Kankaanrinta} 
\givenname{Marja}
\surname{Kankaanrinta}
\address{Department of Mathematics\\University of Virginia \\\newline
         Charlottesville, VA 22903\\USA}
\email{mk5aq@virginia.edu}
\newtheorem{theorem}{Theorem}[section]   
\newtheorem{lemma}[theorem]{Lemma}
\newtheorem{prop}[theorem]{Proposition}         
 
\theoremstyle{definition}
\newtheorem{definition}[theorem]{Definition}

\makeop{Homo}
\numberwithin{equation}{section}

\begin{document}

%\title[On orbifold Riemannian metrics]
%{On real analytic orbifolds and  Riemannian metrics}
%
%\author{Marja Kankaanrinta}
%\address{Department of Mathematics \\
%         P.O. Box 400137\\ University of Virginia\\
%         Charlottesville VA 22904 - 4137\\ U.S.A.}
%\email{mk5aq@@virginia.edu}
%
%
%\date{\today}
%
%\subjclass[2010]{57R18}
%
%
%
%\keywords{orbifold, real analytic, complete, Riemannian metric, frame bundle}
%
%%\maketitle
%%\tableofcontents

\begin{abstract}  
It is known that every  ${\rm C}^r$-orbifold, 
$1\leq r\leq\infty$, 
has a compatible ${\rm C}^s$-differential structure, for every $s$,
where $r< s\leq\omega$.  We prove that if two reduced
${\rm C}^r$-orbifolds, $2\leq r\leq\omega$, are  ${\rm C}^2$-diffeomorphic,
then they are  ${\rm C}^r$-diffeomorphic. It follows that the compatible
${\rm C}^s$-differential structure on a reduced ${\rm C}^r$-orbifold, $2\leq r<s\leq\omega$,
is unique up to
a  ${\rm C}^s$-diffeomorphism.
\end{abstract}

\maketitle
\section{Introduction}

\noindent  In this note we study differentiable structures on orbifolds.  An orbifold is called reduced, if the actions of the local groups on the orbifold charts are effective. Orbifolds that are not reduced
can be diffeomorphic even if their orbifold atlases are quite different. Therefore, we only consider reduced orbifolds. We prove the following result:

%Let $X$ and $Y$ be ${\rm C}^r$-orbifolds, $1\leq r\leq\omega$. If both $X$ and $Y$ are reduced %orbifolds, then $X$ and $Y$ are ${\rm C}^r$-diffeomorphic if and only if they have equivalent %orbifold ${\rm C}^r$-atlases. If at least one of $X$ and $Y$ is not reduced, then $X$ and $Y$ may %be ${\rm C}^r$-diffeomorphic, even if their orbifold ${\rm C}^r$-atlases are not equivalent.
%We therefore introduce the concept of a {\it strong ${\rm C}^r$-diffeomorphism} between ${\rm C}^r%$-orbifolds. Two (not necessarily reduced) ${\rm C}^r$-orbifolds are strongly ${\rm C}^r$-%diffeomorphic if and only if their orbifold ${\rm C}^r$-atlases are equivalent.

%The aim of this  note is to prove the following result:

\begin{theorem}
\label{mainthm}
Let $X$ and $Y$ be  reduced ${\rm C}^r$-differentiable orbifolds,
$2\leq r\leq\omega$. 
 If $X$ and $Y$ are  ${\rm C}^2$-diffeomorphic, then they are
 ${\rm C}^r$-diffeomorphic.
\end{theorem}

As usual, ${\rm C}^\infty$ and ${\rm C}^\omega$ mean smooth and real analytic, respectively.
A {\it ${\rm C}^r$-differentiable structure} on an orbifold $X$ means a maximal
${\rm C}^r$-atlas $\alpha$ on $X$. A ${\rm C}^s$-differentiable structure $\beta$ on $X$, $s>r$,
is called {\it compatible} with $\alpha$, if $\beta\subset\alpha$. In this case, every chart on
$\beta$ is a chart on $\alpha$. For a reduced orbifold $X$, this means 
equivalently that the identity map
on $X$ is a  ${\rm C}^r$-orbifold diffeomorphism $X(\alpha)\to X(\beta)$.
In \cite{Ka} (Theorems 7.4 and 8.2), we proved that every 
${\rm C}^r$-orbifold $X$,  $1\leq r\leq\infty$,
has a compatible ${\rm C}^s$-differential structure, for any $s$ such that $r<s\leq\omega$.

Let $X$ and $Y$ be reduced orbifolds equipped with ${\rm C}^s$-differential
structures $\alpha$ and $\beta$, respectively. Assume there is
a ${\rm C}^s$-diffeomorphism $f\colon X\to Y$. Then $\alpha$ has a
refinement $\alpha_0$ such that $f$ takes $\alpha_0$ to a refinement
$f(\alpha_0)$ of $\beta$ (Theorem \ref{samat3}).
Thus the existence result in \cite{Ka} together with
Theorem \ref{mainthm} imply the following:

\begin{theorem}
\label{unique}
Let $\alpha$ be a ${\rm C}^r$-differential structure on a reduced orbifold
$X$, $2\leq r\leq\infty$. There is a 
 ${\rm C}^s$-differentiable structure $\beta$ on
 $X$ compatible with $\alpha$, for every $s$, $r<s\leq\omega$, and
 $\beta$ is unique up to a  ${\rm C}^s$-diffeomorphism.
\end{theorem}

The proof of Theorem \ref{mainthm} is based on using the frame bundle construction
for reduced orbifolds: The fact that the general linear group ${\rm GL}_n({\mathbb{R}})$ acts
properly on the frame bundle ${\rm Fr}(X)$ of a reduced $n$-dimensional orbifold allows us to
use approximation results (see \cite{IK}) for differentiable equivariant maps between the frame bundles of two reduced orbifolds.

{{\it Acknowledgements.} 
The author's research was supported by the visitors program of the
Aalto University during the academic year 2012 - 2013. The author would
like to thank the Department of  Mathematics and Systems Analysis 
of the Aalto University for its hospitality during her stay.}

%\begin{theorem}
%\label{existence}
%Let $\alpha$ be a ${\rm C}^r$-differential structure, $1\leq r\leq\infty$,
%on the orbifold $X$. 
%Then there exists a compatible ${\rm C}^s$-differential structure
%$\beta\subset\alpha$ on $X$,
%for any $s$ such that $r<s\leq\omega$.
%\end{theorem}

%\begin{theorem} 
%\label{combanal2} Let $1\leq r\leq \infty$ and let $r<s\leq\omega$.
%Then every ${\rm C}^r$-orbifold is ${\rm C}^r$-diffeomorphic 
%to a ${\rm C}^s$-orbifold.
%\end{theorem}

\section{Definitions}
\label{defi}

\noindent We begin with the definition of an orbifold:

\begin{definition}
\label{orbichartdef}
Let $X$ be a topological space and let $n\in {\mathbb{N}}$.
{\begin{enumerate}

\item An {\it orbifold chart} of $X$   is a triple $(\tilde{U}, G, \varphi)$,
where $\tilde{U}$ is a  connected open subset of ${\mathbb{R}}^n$,
$G$ is a finite group acting  on $\tilde{U}$ and
$\varphi\colon \tilde{U}\to X$ is a $G$-invariant map that induces
a homeomorphism 
$U=\varphi(\tilde{U})\cong\tilde{U}/G$. The subgroup
of $G$ acting trivially on $\tilde{U}$ is denoted by ${\rm ker}(G)$.

\item An {\it embedding} $(\lambda,\theta)\colon (\tilde{U}, G, \varphi)\to
(\tilde{V}, H,\psi)$ between two orbifold charts is  an injective
homomorphism $\theta\colon G\to H$ such that $\theta$ is an
isomorphism from ${\rm ker}(G)$ to ${\rm ker}(H)$, and
an  equivariant embedding
$\lambda\colon \tilde{U}\to\tilde{V}$ with $\psi\circ\lambda=\varphi$.
 
 \item An {\it orbifold atlas} on $X$ is a family ${\cal{U}}=\{ (\tilde{U}, G,\varphi)\}$
of orbifold charts  covering $X$ and satisfying the following: For any two
charts $(\tilde{U},G,\varphi)$ and $(\tilde{V}, H,\psi)$ and for any point $x\in
\varphi(\tilde{U})\cap \psi(\tilde{V})$, there exists a chart $(\tilde{W}, K,\mu)$ 
such that $x\in\mu(\tilde{W})$ and
embeddings $(\tilde{W}, K,\mu)\to (\tilde{U},G,\varphi)$ and
$(\tilde{W}, K,\mu)\to (\tilde{V}, H,\psi)$.

\item An orbifold atlas ${\cal{U}}$ {\it refines} another orbifold atlas
${\cal{V}}$ if every chart in ${\cal{U}}$ admits an embedding into some
chart in ${\cal{V}}$. Two orbifold atlases are called 
 {\it equivalent} if they have a common
refinement.
\end{enumerate}}
\end{definition}

\begin{definition}
\label{orbidef}
An $n$-dimensional  {\it orbifold} is a paracompact Hausdorff space
$X$ equipped with an equivalence class of 
$n$-dimensional orbifold atlases.
\end{definition}

An orbifold is called {\it reduced} if $G$ acts effectively on $\tilde{U}$,
for every orbifold chart $(\tilde{U}, G,\varphi)$.

An orbifold is called a {\it ${\rm C}^r$-orbifold}, $1\leq r\leq\omega$,
where ${\rm C}^\infty$ means
{\it smooth} and ${\rm C}^\omega$ means {\it real analytic}, if
for every orbifold chart $(\tilde{U}, G,\varphi)$, the action of $G$ 
on $\tilde{U}$ is ${\rm C}^r$-differentiable,
and if each 
$\lambda\colon\tilde{U}\to\tilde{V}$ is 
a ${\rm C}^r$-embedding.

We recall the definition of an orbifold map:

\begin{definition}
\label{orbimap}
Let $X$ and $Y$ be ${\rm C}^r$-orbifolds, $1\leq r\leq\omega$.
 Let $0\leq p\leq r$. We call a map
$f\colon X\to Y$ a {\it  ${\rm C}^p$-differentiable orbifold map}, 
if for every $x\in X$, there
are charts $(\tilde{U}, G,\varphi)$ around $x$ and $(\tilde{V}, H,\psi)$
around $f(x)$, such that $f$ maps $U=\varphi(\tilde{U})$ into
$V=\psi(\tilde{V})$ and the restriction $f\vert U$
can be lifted to a ${\rm C}^p$-differentiable
equivariant map
$\tilde{f}\colon\tilde{U}\to \tilde{V}$.
\end{definition}

\begin{definition}
\label{diffeo}
Let $X$ and $Y$ be ${\rm C}^r$-orbifolds, $1\leq r\leq\omega$.
Let $1\leq p\leq r$.
A map $f\colon X\to Y$ is called a {\it  ${\rm C}^p$-diffeomorphism},
if $f$ is a ${\rm C}^p$-differentiable bijection, and if the inverse
map $f^{-1}\colon Y\to X$ is ${\rm C}^p$-differentiable. If there is
a ${\rm C}^p$-diffeomorphism $X\to Y$, then we call $X$ and $Y$
{\it ${\rm C}^p$-diffeomorphic}.
\end{definition}

\begin{prop}
\label{strongdiffeo}
Let $X$ and $Y$ be reduced ${\rm C}^r$-orbifolds, $1\leq r\leq\omega$.
Let $1\leq p\leq r$. Let $f\colon X\to Y$ be a bijection. Then the following are equivalent:
{\begin{enumerate}
\item The bijection $f$ is a ${\rm C}^p$-diffeomorphism.
\item For every $x\in X$, there are orbifold charts $(\tilde{U}, G, \varphi)$
of $X$ and $(\tilde{V},H, \psi)$ of $Y$
satisfying the following conditions:
{\begin{enumerate}
\item $f(U)=V$, where $U=\varphi(\tilde{U})$ and
$V=\psi(\tilde{V})$.

\item $x=\varphi(\tilde{x})$ and $f(x)=\psi(\tilde{y})$, for some
$\tilde{x}\in \tilde{U}$ and $\tilde{y}\in\tilde{V}$, respectively.

\item $G_{\tilde{x}}=G$, $H_{\tilde{y}}=H$.

\item  The restriction $f\vert U\colon U\to V$ has 
an equivariant  lift
$\tilde{f}\colon \tilde{U}\to\tilde{V}$ that is a ${\rm C}^p$-diffeomorphism and the
corresponding homomorphism $G\to H$ is an isomorphism.
\end{enumerate}}
\end{enumerate}}
\end{prop}

\begin{proof} 
Clearly, a bijection satisfying Conditions 2 (a) - (d) is a ${\rm C}^p$-diffeomorphism.
Assume then that  $f\colon X\to Y$ is a ${\rm C}^p$-diffeomorphism, and let $x\in X$.
There are orbifold charts $(\tilde{U}, G,\varphi)$ and
$(\tilde{U}', G',\varphi')$ of $X$ around $x$ and $(\tilde{W}, H, \psi)$ and
 $(\tilde{W}', H', \psi')$ of $Y$ around $y=f(x)$, where $\varphi(\tilde{U})=U$,
 $\varphi'(\tilde{U'})=U'$, $\psi(\tilde{W})=W$ and 
 $\psi'(\tilde{W}')=W'$ making the following diagram
 commute:
 
$$ 
 \begin{CD}
\tilde{U} @>\tilde{f}>>\tilde{W} @>\lambda>> \tilde{W}'
@ >\tilde{e}>> \tilde{U}'\\
@VV\varphi V @VV\psi V @VV\psi' V @VV\varphi' V\\
U @>f>> W @>i>> W' @> f^{-1}> >U'
\end{CD}
$$ 
\medskip

\noindent Here $\tilde{f}$ and $\tilde{e}$ are equivariant ${\rm C}^p$-differentiable
lifts of the restrictions $f\vert U$ and $f^{-1}\vert W'$, respectively,
$\lambda\colon \tilde{W}\to\tilde{W}'$ is an equivariant
${\rm C}^r$-embedding and $i\colon W\to W'$ is the inclusion.
We may choose the orbifold charts in such a way that $G=G_{\tilde{x}}$,
for some $\tilde{x}\in\tilde{U}$, where $\varphi(\tilde{x})=x$, and
$H=H_{\tilde{y}}$, for some $\tilde{y}\in\tilde{W}$, where $\psi(\tilde{y})=y$.
Similarly, $G'=G'_{\tilde{x}'}$ for some $\tilde{x}'\in\tilde{U}'$, where
$\varphi'(\tilde{x}')=x$ and $H'=H'_{\tilde{y}'}$ for some 
$\tilde{y}'\in\tilde{W}'$, where $\psi'(\tilde{y}')=y$. 
Since the local groups are unique up to an isomorphism,
it follows that $G'\cong G$ and $H'\cong H$.
According to Lemma 2.2 in \cite{MP}, the composed
map $\tilde{e}\circ \lambda\circ \tilde{f}$ must be an embedding.
It follows that $\tilde{f}$ is an injection and has maximal rank at
every point. Therefore, $\tilde{f}$ is an embedding and $\tilde{f}(
\tilde{U})$ is open in $\tilde{W}$.

Let $\theta\colon  G\to H$ 
be the homomorphism associated with the
lift $\tilde{f}$. Let $g_1, g_2\in G$ and assume $\theta(g_1)=
\theta (g_2)$. Then $\theta(g_1^{-1}g_2)=e$ and 
$\tilde{f}(g_1^{-1}g_2z)=\theta(g_1^{-1}g_2)\tilde{f}(z)=
\tilde{f}(z)$, for all $z\in \tilde{U}$.  Since $\tilde{f}$ is an
injection, it follows that $g_1^{-1}g_2z=z$ for all $z\in \tilde{U}$.
Since $G$ acts effectively on $\tilde{U}$, it follows that
$g_1=g_2$. Therefore, $\theta$ is an injection. Similarly, we can see
that the homomorphism associated with the lift $\tilde{e}$ is injective.
By assumption, the homomorphism associated with $\lambda$ is injective. Since
$G\cong G'$, it follows that $\theta$ 
is an isomorphism. Thus the charts $(\tilde{U}, G,\varphi)$ 
and $(\tilde{V}, H, \psi\vert)$, where
$\tilde{V}=\tilde{f}(\tilde{U})$, satisfy the Conditions $2$ (a), (b), (c) and (d).
\end{proof}

Notice that Proposition \ref{strongdiffeo} does not hold without the assumption that
$X$ and $Y$ are reduced.
The remark on p. 2372 in \cite{Ka2} gives an example of an orbifold diffeomorphism that 
fails to satisfy Condition 2 (d) of Proposition \ref{strongdiffeo}.

Let $X$, $Y$ and $f$ be as in Proposition \ref{strongdiffeo}.
Denote the ${\rm C}^r$-differential structure on $X$ by $\alpha$ and the
${\rm C}^r$-differential structure on $Y$ by $\beta$.  Let
$\alpha_0$ be the collection of the charts $(\tilde{U}, G,\varphi)$ in
$\alpha$ having the property that there is a chart $(\tilde{V}, H, \psi)$ in $\beta$, where
$(\tilde{U}, G, \varphi)$ and $(\tilde{V}, H, \psi)$ satisfy Conditions
2 (a) - (d) of  Proposition \ref{strongdiffeo}.
Thus, for every chart in $\alpha_0$ we associate a chart in
$\beta$. We denote by $f(\alpha_0)$ the collection of charts in
$\beta$ obtained in this way.

\begin{theorem}
\label{samat3}
Let $X$ and $Y$ be reduced  ${\rm C}^r$-orbifolds, $1\leq r\leq\omega$, and let
$f\colon X\to Y$ be a  ${\rm C}^r$-diffeomorphism. Let $\alpha$ and
$\beta$ be the ${\rm C}^r$-differential structures on $X$ and $Y$, respectively.
Then:
{\begin{enumerate}
\item The collection $\alpha_0$ is a ${\rm C}^r$-atlas on $X$ refining $\alpha$.
\item The collection $f(\alpha_0)$ is a ${\rm C}^r$-atlas on $Y$ refining $\beta$.
\end{enumerate}}
\end{theorem}

\begin{proof}
Let $(\tilde{U}_i, G_i, \varphi_i)\in \alpha_0$ for $i=1,2$, and let $x\in\varphi_1(\tilde{U}_1)\cap
\varphi_2(\tilde{U}_2)$. Let $(\tilde{V}_i, G_i, \psi_i)\in f(\alpha_0)$, $i=1,2$, be the corresponding charts in $\beta$. Let $\tilde{x}_i\in \tilde{U}_i$ and $\tilde{y}_i\in \tilde{V}_i$ be such that
$G_i=G_{\tilde{x}_i}=G_{\tilde{y}_i}$, for $i=1,2$, where $\tilde{y}_i=\tilde{f}_i(\tilde{x}_i)$ and
$\tilde{f}_i\colon \tilde{U}_i\to \tilde{V}_i$ is an equivariant lift of $f\vert\colon \varphi_i(
\tilde{U}_i)\to \psi_i(\tilde{V}_i)$ as in Condition 2 (d) of Proposition \ref{strongdiffeo}.

There is a chart $(\tilde{U},G,\varphi)\in
\alpha$ such that $x\in \varphi(\tilde{U})$ and embeddings $\lambda_i\colon
(\tilde{U}, G,\varphi)\to (\tilde{U}_i,G_i,\varphi_i)$, for $i=1,2$. 
Let $\tilde{x}\in\tilde{U}$ be such that $\varphi(\tilde{x})=x$. We may assume that
$G_{\tilde{x}}=G$.
Let $\theta_i\colon
G\to G_i$ denote the  injective homomorphisms
associated with the embeddings $\lambda_i$. Then
$\tilde{f}_i\circ\lambda_i\colon \tilde{U}\to\tilde{V}_i$, is a $\theta_i$-equivariant embedding
and $(\tilde{f}_i(\lambda_i(\tilde{U})), G, \psi_i\vert)$ is an orbifold chart of $Y$, for
$i=1,2$. Now, $f(x)=f(\varphi(\tilde{x}))\in \psi_1(\tilde{f}_1(\lambda_1(\tilde{U})))\cap
 \psi_2(\tilde{f}_2(\lambda_2(\tilde{U})))$. Let $\tilde{z}_i=\tilde{f}_i(\lambda_i(\tilde{x}))$,
 for $i=1,2$. There is a chart $(\tilde{V}, G,\psi)$ of $Y$ 
 such that $f(x)\in \psi(\tilde{V})$,  with embeddings
 $\mu_i\colon (\tilde{V},G,\psi)\to (\tilde{f}_i(\lambda_i(\tilde{U})), G, \psi_i\vert)$.
 Let $\tilde{z}\in\tilde{V}$ be such that  $\mu_i(\tilde{z})=
 \tilde{z}_i$. Then $G_{\tilde{z}}=G$. The charts 
 $(\lambda^{-1} (\tilde{f}_1^{-1}(\mu_1(\tilde{V}))), G, \varphi\vert)$
 and $(\tilde{V}, G,\psi)$, of $X$ and $Y$, respectively, satisfy the Conditions 2 (a) - (d) of
 Proposition \ref{strongdiffeo}. Consequently, 
 $(\lambda^{-1} (\tilde{f}_1^{-1}(\mu_1(\tilde{V}))), G, \varphi\vert)\in\alpha_0$
 and $(\tilde{V}, G,\psi)\in f(\alpha_0)$.
\end{proof}

\section{The Proofs}
\label{redcase}

\noindent We first recall some well-known
facts having to do with {\it quotient orbifolds}, for more details see
\cite{Ka2}, Section 3.
Let $G$ be a Lie group and let $M$  and $N$ be  real analytic manifolds.
Assume $G$ acts
on $M$  and $N$, respectively,  by  proper,  effective, almost free, real analytic  actions. 
%An almost free action means an action such that all the isotropy subgroups are finite.
(An action of $G$ on $M$ is proper if the map $G\times M\to M\times M$,
$(g,x)\mapsto (x,gx)$, is proper. It is almost free if all the isotropy subgroups are finite.)
Then the orbit space $M/G$ is a reduced real analytic orbifold.
The orbifold charts of $M/G$ are the triples $({\rm N}_x, G_x, \pi_x)$,
where $x\in M$, ${\rm N}_x$ is a {\it linear slice} at $x$, $G_x$ is the isotropy
subgroup at $x$ and $\pi_x\colon {\rm N}_x\to {\rm N}_x/G_x\cong
(G{\rm N}_x)/G$ is the natural projection. 
%Let then $G$ be Lie group and let $M$ and $N$ be
%real analytic manifolds. Assume $G$ acts on $M$ and $N$ by
%proper,  effective, almost free, real analytic actions. 
Every
$G$-equivariant ${\rm C}^r$-diffeomorphism
$f\colon M\to N$, $1\leq r\leq\omega$,  induces 
a ${\rm C}^r$-diffeomorphism $\tilde{f}\colon
M/G\to N/G$.

We next recall the definition and some properties of the {\it frame bundle} of a
reduced real analytic orbifold. For proofs and details, see \cite{MM}, pp. 42--43.
In \cite{MM}, the frame bundle is constructed for reduced smooth
orbifolds, but the same construction goes through in the
real analytic case. 
%Notice that  to construct the frame bundle for a real analytic 
%orbifold $X$ it is not necessary to equip $X$ by a real analytic Riemannian metric
%(which exists by \cite{Ka2}) but it suffices to equip every chart $(\tilde{U}, G, \varphi)$
%by a real analytic $G$-invariant Riemannian metric.

Let $X$ be a reduced real analytic orbifold of dimension $n$ and let
$$
{\mathcal{U}}=\{ (\tilde{U}_i, G_i, \varphi_i)\}_{i\in I}
$$
be the maximal orbifold atlas of $X$.
We first construct the frame bundle
${\rm Fr}(\tilde{U}_i)$ for every chart $\tilde{U}_i$.
The action of $G_i$ on $\tilde{U}_i$ lifts to a
left action on ${\rm Fr}(\tilde{U}_i)$: $g(x,B)=(gx,(dg)_x\circ B)$,
for every $(x,B)\in {\rm Fr}(\tilde{U}_i)$. This action is free and it commutes with
the right action of  the general linear group
${\rm GL}_n(\mathbb{R})$. In particular, ${\rm Fr}(
\tilde{U}_i)/G_i$ is a real analytic manifold on which ${\rm GL}_n({\mathbb{R}})$
acts from the right by a proper, effective,
 almost free, real analytic action, and we may identify
${\rm Fr}(\tilde{U}_i)/G_i$ with the twisted product 
$\tilde{U}_i\times_{G_i}{\rm GL}_n({\mathbb{R}})$.
Let
$p_i\colon {\rm Fr}(\tilde{U}_i)/G_i\to \tilde{U}_i/G_i$ denote the
natural projection.

Let $\lambda\colon (\tilde{U}_i, G_i, \varphi_i)\to (\tilde{U}_j,G_j,\varphi_j)$
be a real analytic embedding between orbifold charts. Let $\theta\colon
G_i\to G_j$ be the homomorphism associated with $\lambda$.
Then  $\lambda$, together with the
differential $d\lambda$, induces an embedding 
$\tilde{\lambda}=(\lambda, d\lambda)\colon
{\rm Fr}(\tilde{U}_i)\to {\rm Fr}(\tilde{U}_j)$. 
Since $\lambda$ is $\theta$-equivariant, it follows that
this embedding factors as
$$
\lambda_{\ast}\colon {\rm Fr}(\tilde{U}_i)/G_i\to
{\rm Fr}(\tilde{U}_j)/G_j.
$$
%Indeed, for any $g\in G_i$ and $e\in{\rm T}_x(\tilde{U}_i)$, we have
%$$
%\begin{array}{lcl}
%\tilde{\lambda}(ge) & = & (d\lambda)_{gx}\circ
%(dg)_x\circ e\\
%&=& d(\lambda\circ g\circ \lambda^{-1})_{\lambda(x)}\circ
%(d\lambda)_x\circ e\\
%&=& d(\bar{\lambda}(g))_{\lambda(x)}\circ\tilde{\lambda}(e)\\
%&=&\bar{\lambda}(g)\tilde{\lambda}(e).
%\end{array}
%$$
The map
$\lambda_{\ast}$ is a real analytic open embedding,
it commutes with the action of
${\rm GL}_n({\mathbb{R}})$ and $p_j\circ \lambda_{\ast}=p_i$. 
%In particular, $g_{\ast}=
%{\rm id}$, which implies that for any two embeddings 
%$\lambda,\mu\colon (\tilde{U}_i, G_i, \varphi_i)\to (\tilde{U}_j,G_j,\varphi_j)$ we have
%$\lambda_{\ast}=\mu_{\ast}$.

The manifolds ${\rm Fr}(\tilde{U}_i)/G_i$, for all $i\in I$,  together with  
the real analytic embeddings $\lambda_{\ast}$ induced by 
all the embeddings $\lambda$ between the orbifold charts, form a
filtered direct system. The {\it frame bundle} ${\rm Fr}(X)$ 
of the real analytic orbifold $X$ is defined to
be the colimit of this system,
$$
{\rm Fr}(X)=\varinjlim\{ {\rm Fr}(\tilde{U}_i)/G_i,\lambda_{\ast}\}.
$$
Then ${\rm Fr}(X)$ is a real analytic manifold, each
${\rm Fr}(\tilde{U}_i)/G_i$ is canonically embedded into
${\rm Fr}(X)$ as an open submanifold and  the maps $p_i$ induce
an open map $p\colon {\rm Fr}(X)\to X$. The following theorem lists some of the 
properties of ${\rm Fr}(X)$:

\begin{theorem}
\label{frame}
Let $X$ be a reduced real analytic orbifold of
dimension $n$ and let
${\rm Fr}(X)$ be the frame bundle of $X$. Then 
${\rm Fr}(X)$ is a real analytic manifold. The general
linear group ${\rm GL}_n({\mathbb{R}})$ acts on
${\rm Fr}(X)$ by a proper,   effective,
almost free, real analytic action. The quotient orbifold ${\rm Fr}(X)/
{\rm GL}_n({\mathbb{R}})$ is real analytically 
diffeomorphic to $X$.
\end{theorem}

\begin{proof}
The other parts of the claim except properness are explained in 
\cite{MM}, pp. 42 - 43, in the case of smooth orbifolds. 
The real analytic case is similar. Since the action of
${\rm GL}_n({\mathbb{R}})$ on ${\rm Fr}(X)$ is obviously Cartan 
(Definition 1.1.2 in \cite{Pa}) and since ${\rm Fr}(X)/{\rm GL}_n({\mathbb{R}})$ 
is regular, it follows from Proposition 1.2.5 in \cite{Pa} that the
action is proper.
\end{proof}

\begin{lemma} 
\label{redapu}
Let $X$ and $Y$ be $n$-dimensional
reduced real analytic orbifolds, and
let $f\colon X\to Y$ be a ${\rm C}^r$-diffeomorphism, $2\leq r\leq\omega$.
Then $f$ induces a ${\rm GL}_n({\mathbb{R}})$-equivariant
${\rm C}^{r-1}$-diffeomorphism $\hat{f}\colon {\rm Fr}(X)\to {\rm Fr}(Y)$, and the diagram
$$
\begin{CD}
{\rm Fr}(X) @>\hat{f}>> {\rm Fr}(Y)\\
@VVp_XV  @VVp_YV\\
X @>f>> Y
\end{CD}
$$
commutes.
\end{lemma}

\begin{proof}
Let $\alpha$ and $\beta$ be the real analytic differential
structures of $X$ and $Y$, respectively. By  Theorem
\ref{samat3}, $\alpha$ and $\beta$ have
refinements $\alpha_0$ and $f(\alpha_0)$, respectively, with
the property that
for every chart $(\tilde{U}_i,G_i,\varphi_i)$ in $\alpha_0$ there is a
chart $(\tilde{V}_i, G_i,\psi_i)$ in $f(\alpha_0)$ such that the
charts $(\tilde{U}_i,G_i,\varphi_i)$  and $(\tilde{V}_i, G_i,\psi_i)$
satisfy the Conditions 2 (a) - (d) of Proposition \ref{strongdiffeo}. Denote the restrictions of
$f$ to $\varphi_i(\tilde{U}_i)$ by $f_i\colon \varphi_i(\tilde{U}_i)\to
\psi_i(\tilde{V}_i)$. Then $f_i$ has a lift $\tilde{f}_i\colon
\tilde{U}_i\to\tilde{V}_i$, and  $\tilde{f}_i$ is
a $G_i$-equivariant  ${\rm C}^r$-diffeomorphism. The ${\rm C}^{r-1}$-maps
$(\tilde{f}_i, d\tilde{f}_i)\colon \tilde{U}_i\times {\rm GL}_n(
{\mathbb{R}})\to \tilde{V}_i\times {\rm GL}_n({\mathbb{R}})$
induce ${\rm GL}_n({\mathbb{R}})$-equivariant ${\rm C}^{r-1}$-diffeomorphisms
$$
\bar{f}_i\colon {\rm Fr}(\tilde{U}_i)/G_i\cong
\tilde{U}_i\times_{G_i}{\rm GL}_n({\mathbb{R}})\to
\tilde{V}_i\times_{G_i}{\rm GL}_n({\mathbb{R}})
\cong {\rm Fr}(\tilde{V}_i)/G_i.
$$
Let $\lambda\colon\tilde{U}_i\to\tilde{U}_j$ and
$\mu\colon \tilde{V}_i\to\tilde{V}_j$ be embeddings. Then
$\tilde{f}_j^{-1}\circ \mu\circ\tilde{f}_i\colon \tilde{U}_i\to\tilde{U}_j$
is also an embedding. By Proposition A.1 in \cite{MP}, there is 
a unique $g\in G_i$
such that $\lambda=g\circ \tilde{f}_j^{-1}\circ \mu\circ\tilde{f}_i$.
Therefore, $\tilde{f}_j\circ \lambda= g\circ \mu\circ\tilde{f}_i$.
It follows that the maps $\bar{f}_i$ commute with the embeddings 
$$
\lambda_\ast
\colon {\rm Fr}(\tilde{U}_i)/G_i\to {\rm Fr}(\tilde{U}_j)/G_j\,\,
{\rm and} \,\, \mu_\ast \colon {\rm Fr}(\tilde{V}_i)/G_i\to {\rm Fr}(\tilde{V}_j)/G_j
$$
that are used to define the frame bundles ${\rm Fr}(X)$ and
${\rm Fr}(Y)$, respectively.
Thus they induce a
${\rm GL}_n({\mathbb{R}})$-equivariant 
${\rm C}^{r-1}$-diffeomorphism $\hat{f}\colon {\rm Fr}(X)\to
{\rm Fr}(Y)$.
\end{proof}

Notice that in Theorem \ref{mainthm} we assume that $2\leq r$ instead of $1\leq r$.  
The reason for doing so is that  the proof of Theorem \ref{mainthm}
uses Lemma \ref{redapu}. 
%The maps 
%$\bar{f}_i$ in Lemma \ref{redapu} are induced by the maps $(\tilde{f}_i,
%d\tilde{f}_i)$. Hence the $\bar{f}_i$ are ${\rm C}^{r-1}$ if the $\tilde{f}_i$ are ${\rm C}^r$.

\medskip 

\noindent{\it Proof of Theorem \ref{mainthm}.}
Let $X$ and $Y$ be  reduced ${\rm C}^r$-orbifolds, $2\leq r\leq\omega$, and
let $f\colon X\to Y$ be a ${\rm C}^2$-diffeomorphism. According to Theorem 8.4 in
\cite{Ka}, there are real analytic orbifolds $X^\omega$ and $Y^\omega$ and
${\rm C}^r$-diffeomorphisms $f_1\colon X\to X^\omega$ and $f_2\colon
Y\to Y^\omega$. By Theorem \ref{frame}, there are real analytic diffeomorphisms
$h_1\colon X^\omega\to {\rm Fr}(X^\omega)/{\rm GL}_n({\mathbb{R}})$ and
$h_2\colon Y^\omega\to {\rm Fr}(Y^\omega)/{\rm GL}_n({\mathbb{R}})$.
Let $g= f_2\circ f\circ f_1^{-1}\colon X^\omega\to Y^\omega$. 
Then $g$ is a
${\rm C}^2$-diffeomorphism, and 
by Lemma \ref{redapu} it induces
a ${\rm GL}_n({\mathbb{R}})$-equivariant ${\rm C}^1$-diffeomorphism
$\hat{g}\colon {\rm Fr}(X^\omega)\to {\rm Fr}(Y^\omega)$.
Now, $\hat{g}$ induces a ${\rm C}^1$-diffeomorphism
$\tilde{g}\colon {\rm Fr}(X^\omega)/{\rm GL}_n({\mathbb{R}}) \to
{\rm Fr}(Y^\omega)/{\rm GL}_n({\mathbb{R}})$ and the  diagram 

$$
\begin{CD}
{\rm Fr}(X^\omega) @>\hat{g}>> {\rm Fr}(Y^\omega)\\
@VV\pi_1V  @VV\pi_2V\\
{\rm Fr}(X^\omega)/{\rm GL}_n({\mathbb{R}})
@>\tilde{g}>>
{\rm Fr}(Y^\omega)/{\rm GL}_n({\mathbb{R}})\\
@VVh_1^{-1}V @VVh_2^{-1}V\\
X^\omega @>g>> Y^\omega\\
@VVf_1^{-1}V   @VVf_2^{-1}V\\
X @>f>> Y
\end{CD}
$$

commutes. By Corollary IIa in \cite{IK}, there is a 
${\rm GL}_n({\mathbb{R}})$-equivariant
real analytic diffeomorphism $\hat{h}\colon
 {\rm Fr}(X^\omega)\to {\rm Fr}(Y^\omega)$. Then
 $\hat{h}$ induces a real analytic diffeomorphism
 $\tilde{h}\colon 
 {\rm Fr}(X^\omega)/{\rm GL}_n({\mathbb{R}}) \to
{\rm Fr}(Y^\omega)/{\rm GL}_n({\mathbb{R}})$.
It follows that $h=h_2^{-1}\circ \tilde{h}\circ h_1\colon
X^\omega\to Y^\omega$ is a real analytic diffeomorphism.
Therefore, $h^r=f_2^{-1}\circ h\circ f_1\colon
X\to Y$ is a ${\rm C}^r$-diffeomorphism. 
\qed

\end {document}